\documentclass[a4paper]{article}
\usepackage{amsmath,amssymb,amsthm,mathrsfs,mathrsfs,amsfonts}
\usepackage{algorithm,algorithmic,enumitem}
\usepackage[colorlinks,linkcolor=blue,citecolor=blue,urlcolor=blue]{hyperref}
\usepackage{multirow}
\usepackage{xcolor}
\usepackage{authblk}
\usepackage[scale=0.7]{geometry}

\usepackage{cite}
\usepackage[utf8]{inputenc}
\usepackage{multicol}
\usepackage{multirow}
\usepackage[numbers]{natbib}
\usepackage{graphicx}
\usepackage{siunitx}
\usepackage{subcaption}
\usepackage{todonotes}
\usepackage{url}
\usepackage{bm}
\usepackage{xspace}
\usepackage{appendix}
\usepackage{cite}
\usepackage{epstopdf}
 \usepackage{epsfig}

\newcommand\EX{\mathbb{E}}

\newcommand{\be}{\begin{equation}}
\newcommand{\ee}{\end{equation}}

\numberwithin{equation}{section}

\newtheorem{theorem}{Theorem}
\newtheorem{lemma}{Lemma}

\newtheorem{corollary}{Corollary}
\newtheorem{remark}{Remark}

\providecommand{\keywords}[1]{\textbf{\textit{Keywords:}} #1}

\title{\bf \Large A convergence  study of SGD-type methods for stochastic optimization}
\author{Tiannan Xiao,~Guoguo Yang}
\affil{\small LMAM and School of Mathematical Sciences, Peking University, Beijing 100871, China }
\affil{alxeusxiao@pku.edu.cn,ygj512@hotmail.com}

\date{}

\begin{document}
\maketitle
\begin{abstract}
In this paper, we first reinvestigate  the convergence of vanilla SGD method in the sense of $L^2$  under  more general learning rates conditions and a more general convex assumption,  which relieves the conditions on learning rates and  do not need the problem to be strongly convex.
Then, by taking advantage of the Lyapunov function technique, we present the convergence of the momentum SGD and Nesterov accelerated SGD methods for the  convex and  non-convex problem under $L$-smooth assumption that extends the bounded gradient limitation to a certain extent.
 The convergence of time averaged SGD was also analyzed.
\end{abstract}


\section{Introduction}
\label{sec1} In this article, we study  the convergence analysis of stochastic gradient descent (SGD) type methods to  the optimization problem
\begin{equation}\label{eq:opt1}
\begin{split}
\min_{x\in \mathbb{R}^d} f(x) := \frac{1}{S}\sum_{i=1}^S f_i(x),
\end{split}
\end{equation}
where $f, f_i:\mathbb{R}^d \rightarrow \mathbb{R}$ are continuously differentiable functions and $S$ is the number of samples in machine learning. Recently, stochastic gradient descent (SGD) has played a significant role in training machine learning models when $S$ is very large and $x$ has many components. The SGD is derived from gradient descent by replacing $\nabla f$ with $\nabla f_{s_k}$, where $s_k$ is a random variable uniformly sampled from $\{1,2,\ldots,S\}$. The iterative format is often read as
{\begin{equation}\label{eq:SGD1}
\begin{split}
x_{k} & = x_{k-1} - \alpha_k \nabla f_{s_k}(x_{k-1})\\
 & = x_{k-1} - \alpha_k \nabla f(x_{k-1})+\alpha_k\xi_k,
\end{split}
\end{equation}}
where $\alpha_k$ is the learning rate, which satisfies the assumption (Divergence condition):
\begin{equation}\label{Assum1:alpha}
\lim_{k\rightarrow \infty}\alpha_k=0,\quad\sum_{k=1}^\infty \alpha_k=\infty.
\end{equation}
{In \eqref{eq:SGD1}, the term $\xi_k=\nabla f(x_{k-1})-\nabla f_{s_k}(x_{k-1})$.
Let $\mathcal{F}_k = \sigma(x_0,\xi_1,\xi_2,\cdots,\xi_k)$ be the filtration generated by $(x_0,\xi_1,\ldots, \xi_k)$,
thus $\xi_k$  satisfies  $\mathbb{E} [\xi_k | \mathcal{F}_{k-1}] = 0$.}

 For iterative format (\ref{eq:SGD1}), it has a mini-batch SGD \cite{Goodfellow2016}  variant, which utilises $\frac{1}{m} \sum_{i=1}^m\nabla f_{s_{k_i}}(x_k)$ to estimate gradient, where $s_{k_i}$ are i.i.d random variables uniformly sampled from $\{1,2,\ldots,S\}$ and the noise term $\xi_k = \nabla f(x_{k-1}) - \frac{1}{m} \sum_{i=1}^m\nabla f_{s_{k_i}}(x_k).$  For convenience, we will choose sample count $m=1$ in this paper, and the results of this paper are consistent for cases where $m>1$.

Many elegant works  have been done on the forms of generalization and theoretical analysis of SGD-type methods \cite{nguyen2018sgd,bottou2012stochastic,li2022revisiting,MGDC}. Here,
the general Markovian iteration forms of SGD-type methods are  denoted as

$~~~~\bullet$ vanilla SGD (vSGD)
\begin{equation}
 x_{k} = x_{k-1} - \alpha_k F(x_{k-1},\xi_k),  \label{eq:vSGD}
\end{equation}

$~~~~\bullet$ momentum SGD (mSGD)  \cite{1964Some}
\begin{equation}\label{eq:mSGD}
\begin{split}
x_{k} & = x_{k-1} +v_k, ~v_k = \beta_k v_{k-1} - \alpha_k F(x_{k-1},\xi_k),
\end{split}
\end{equation}

$~~~~\bullet$  Nesterov accelerated form  (NaSGD) \cite{1983A}
\begin{equation}\label{eq:NaSGD}
\begin{split}
y_k & =  x_k+ \beta_k (x_k -  x_{k-1}), ~x_{k}  = y_{k-1} - \alpha_k F(y_{k-1},\xi_k),
\end{split}
\end{equation}
respectively.
Here $\EX[ F(x_{k-1}, \xi_k)|\mathcal{F}_{k-1}]=\nabla f(x_{k-1})$  and $\beta_k\in [0,1)$ in (\ref{eq:mSGD}) and (\ref{eq:NaSGD}).  For the above mentioned SGD-type methods, we assume  the noise term $\{\xi_k\}$ satisfy  the following conditional mean and covariance conditions
\begin{equation}\label{Assum:StatXi}
\mathbb{E}[\xi_k|\mathcal{F}_{k-1}] = 0,\quad \mathbb{E} [\| \xi_k \|^2|\mathcal{F}_{k-1} ] \le M + V\| \nabla f(x_{k-1})\|^2,
\end{equation}
which covers the usual incremental SGD  \cite{Bertsekas00} and $\mathbb{E}|f(x_0)| < \infty$ for the initial value $x_0$ in this paper.

In recent years, vSGD \eqref{eq:vSGD} has attracted an increasing number of researchers.  Bertsekas and Tsitsiklis \cite{Bertsekas00}  proved that $\{x_k\}$ converges almost surely to a critical point of $f$  when the stepsize $\{\alpha_k\}$ satisfies $\sum \alpha_k = \infty$ and $\sum \alpha_k^2 < \infty$  even for non-convex problems. Ghadimi and Guanghui \cite{SGDZO} analyzed the complexity of $\{x_k\}$ to approximate stationary point of a nonlinear problem and showed that this method is non-asymptotically convergent with $\min_{k \le n}\mathbb{E} \| \nabla f(x_k)\|^2 = O(1/\sqrt{n})$  if the problem is  non-convex.  By using the variance reduction technique,  Reddi et al. \cite{NCSGDP1} showed that the convergence rate can be improved to $O(1/n)$.
For strongly convex {problems}, this convergence rate can be improved to $\mathbb{E}\| \nabla f(x_n)\|^2 = O(1/n)$  when the step size $\alpha_k=O(1/k)$
\cite{BLN,SGDASGD,li2022revisiting,JMLR:v20:18-759}.

 Compared with SGD method,  there are relatively few references about the mSGD and NaSGD.   Barakat and Bianchi \cite{MGDC} presented a novel first order convergence rate result  of a general class mSGD.
Liu, Gao, and Yin  \cite{SGDB} established  the stationary convergence bound of this time averaged mSGD when the step sizes are  constant.
For Nesterov acceleration gradient,  Su, Boyd, and  Cand{{\`e}}s \cite{NAG1} showed that the convergence rate of $f(x_n)$ towards $f(x^*)$ is $O(n^{-2})$ in {the} deterministic case when $f$ is convex.  Assran and Rabbat  \cite{NASGDB} studied the stationary convergence bound of NaSGD with constant step size.

Besides the aforementioned convergence analysis, the CLT for the SGD-type methods was  studied in \cite{li2022revisiting} under more general divergence condition (\ref{Assum1:alpha}). It is a natural question to study the convergence of SGD-type methods under condition (\ref{Assum1:alpha}) and whether we can generalize this condition. There are few researches on the convergence of SGD-type methods under this condition. This is a motivation for this work, to give a further  investigation  on the convergence analysis for the SGD-type methods. We will consider the convergence analysis of the vSGD \eqref{eq:vSGD}, NaSGD \eqref{eq:NaSGD}, and the mSGD as follows
\begin{equation}\label{eq:mSGD2}
\begin{split}
x_{k} & = x_{k-1} + \alpha_k v_k,~v_k = (1 - \beta_k) v_{k-1} - \alpha_k F(x_{k-1},\xi_k),
\end{split}
\end{equation}
where $\beta_k=\mu_k\alpha_k$ and $\mu_k>0$ is the damping parameter. We consider the form \eqref{eq:mSGD2} instead of \eqref{eq:mSGD} mainly because it has a better connection to the continuous time limit.

The main predominant contributions of this paper are as follows.

$~~~~\bullet$  \textit{Convergence of vSGD.} We investigate the convergent analysis of vSGD with different {setups} on
the assumptions on noise and step size.
Compared with the assumptions of the previous article regarding the step size $\{\alpha_k\}$,  such as $\sum \alpha_k = \infty$ and $\sum \alpha_k^2 < \infty$ \cite{Bertsekas00} for convex problem, we  relieve the conditions on $\{\alpha_k\}$ and  do not need $f(x)$ to be strongly convex.

$~~~~\bullet$ \textit{Convergence of  mSGD and NaSGD.} For mSGD,
this part is classified into two cases: the case with constant damping $\mu_k\equiv\tilde{\mu}$, or the case with vanishing damping $\mu_k\rightarrow 0$.    Taking advantage of the Lyapunov function technique, we can show the convergent results for both cases.  Unlike the previous convergence analysis \cite{gitman2019understanding}, which requires the gradient to be bounded, we assume that $f(x)$ in (\ref{eq:opt1})  is $L$-smooth:   there exists $L > 0$, such that
\begin{equation}\label{eq:L-smooth}
\forall x,y\in \mathbb{R}^d, \quad \|\nabla f(x) - \nabla f(y) \| \le L \| x - y \|,
\end{equation}
 which extends the bounded gradient  limitation to a certain extent.
 For NaSGD, its stationary convergence bound was studied in \cite{NASGDB} when the learning rates are constant, which limits its application in practical problems.  Here, we show the convergent results under $L$-smooth (\ref{eq:L-smooth}) and more general learning rates conditions.

$~~~~\bullet$  \textit{Convergence of the time average $x_k$.}  Similar to {\cite{ASSGD, li2022revisiting}},  we consider the time averaged SGD
\begin{equation} \label{eq:aSGD}
\bar{x}_n = \frac{\sum_{k=1}^n \alpha_k x_{k-1}}{\sum_{k=1}^n \alpha_k},
\end{equation}
which is an analogy of the continuous form $\int_0^T x(t)dt /T$, where $T$ is the summation of step size. {Compared with the analysis in \cite{gitman2019understanding}, we extend the bounded gradient  limitation of  $f(x)$ to $L$-smooth (\ref{eq:L-smooth}).}

The rest of this paper is organized as follows. We will prove the convergence of the vSGD, mSGD and NaSGD, and the average $\bar{x}_n$ in Sections \ref{sec:sgd}, \ref{sec:msgd} and \ref{sec:average},  respectively. Finally, we make the conclusion.  In the remainder of this paper, we will use $C$ as a  $O(1)$  positive constant in different estimates, the value of which may vary in different places.

\section{Convergence for vSGD}\label{sec:sgd}
In this section, we will give the convergence analysis of vSGD,  which will relieve  the constraint on step size $\{\alpha_k\}$ in \cite{Bertsekas00} by a more general convex assumption.
The SGD convergence of  non-convex function $f$ is already known. In \cite{Bertsekas00}, it proves that convergence in probability $1$ can be achieved in the sense of the following two limits (\ref{limit1}) and (\ref{limit2}). The more general conclusion we give here will also be encountered in  mSGD in the next section.
It is worth noting that  proving the weak convergence  (\ref{limit1}) does not require the condition $\sum_{k=1}^\infty \alpha_k^2<\infty$.
If some restrictions are imposed on $\{\xi_k \}$, the convergence can be strengthened, for example, variance reduction SGD \cite{SVRG}, SARAH \cite{SARAH} etc. These variants can get exponential convergence for strongly convex function as  Gradient Descent and $O(1/t)$ for non-convex function \cite{2019Finite}.

To obtain the convergent theorems, we first give the following lemmas. The results of (a) and (b) in Lemma \ref{lm:3sq} are similar to those of Lemma 1 in reference \cite{Bertsekas00}. The results of (c) is to prove the convergence of mSGD later.
\begin{lemma}[Sequence convergence]\label{lm:3sq}
Assume the step size $\{\alpha_k\}$  satisfy (\ref{Assum1:alpha}), the positive sequence $Z_k \rightarrow 0$ and $X_k$ has a lower bound $X$.  Consider the triplet $\{X_k,Y_k,Z_k\}_k$  with the relation
\begin{equation*}
\begin{split}
X_k \le X_{k-1} - \alpha_k Y_k + \alpha_k Z_k,~~~~Y_k \ge 0.
\end{split}
\end{equation*}
We have: (a) $\lim\inf _{k\rightarrow \infty}Y_k = 0,$ and there exists $K>0$ such that $\sum_{k=1}^n \alpha_k Y_k \le  K + \sum_{k=1}^n \alpha_k Z_k.$  (b) If $\sum \alpha_k Z_k < +\infty$, then $X_n$ is convergent. (c) Let $\{Y_{n_k} \}$ and $\{X_{n_k} \}$ be subsequences of $\{Y_{n} \}$  and $\{X_{n} \}$, respectively. If $Y_{n_k} \rightarrow 0$ is a sufficient condition  for $X_{n_k}\rightarrow X$, then $\lim_{n \rightarrow \infty} X_n = X.$

\end{lemma}

\begin{proof}  {For the proof of (b), it is a special case of  Lemma 1 in \cite{Bertsekas00}.} 

{
(a) It is easy to know that for $m < n, X_n \le X_{m-1} - \sum_{i=m}^n \alpha_i Y_i + \sum_{i=m}^n \alpha_i Z_i,$ so we have
\begin{equation*}
\begin{split}
\sum_{i=1}^n \alpha_i Y_i \le \sum_{i=1}^n \alpha_i Z_i + X_0 - X_n \le \sum_{i=1}^n \alpha_i Z_i + X_0 - X.
\end{split}
\end{equation*}
If $\lim_{k\rightarrow \infty}\sup Y_k > 0$, it means that there exists $\varepsilon > 0, ~K>0$ such that for all $ k \ge K, Y_k \ge \varepsilon$,
then there exists $ M \ge K,  \text{for}~ m > M, Z_m \le \varepsilon / 2$, and
$$X_n \le X_{m} - \sum_{i=m+1}^n \alpha_i Y_i + \sum_{i=m+1}^n \alpha_i Z_i \le X_{m} - \frac{\varepsilon}{2} \sum_{i=m+1}^n \alpha_i,  \text{~for~all~} n < m,$$ 
thus when $n \rightarrow +\infty$, $X_n \rightarrow - \infty,$  which leads to contradiction as $X_n$ has a lower bound.
}

{
(b) Since $X_n \le X_0 + \sum_{k=0}^{\infty} \alpha_k Z_k$, we get $\{X_n \}$ is bounded. From (a),  we know that for $\varepsilon > 0$, there exists $M >0$ such that
\begin{equation*}
\begin{split}
X_n \le X_m - \sum_{i=m}^n \alpha_i Y_i + \sum_{i=m}^n \alpha_i Z_i \le X_m + \varepsilon,  ~n \ge m > M.
\end{split}
\end{equation*}
If $X_n$ is not convergent, then assume $\lim \sup_k X_k = A_1, \lim \inf_k X_k = A_2$ and take $\varepsilon = (A_1 - A_2) / 3$ . Now we can find infinite $m$ fulfilled $X_m < A_2 + \varepsilon$ with $m > M$,  then there exists $n > m$ such that $X_{n} > A_1 - \varepsilon$.  Then $X_n > X_m + \varepsilon,$  which leads to contradiction.
}

(c) Without loss of generality, we set $X=0$.  From (a), we know $\lim \inf_k X_k = 0$.  If there exists $\varepsilon >0$ such that $\lim \sup_n X_n \geq \varepsilon$, we can find infinite  $k$ such that $X_k < \varepsilon / 2, X_{m_k} \ge \varepsilon,  X_i \in [\varepsilon / 2, \varepsilon),~k< i <m_k$  and there exists $\delta>0$, $Y_i \ge \delta.$  By $k \rightarrow +\infty, Z_k \le \delta / 2$, we have
\begin{equation*}
\begin{split}
\varepsilon \le X_{m_k} \le X_k - \sum_{i=k+1}^{m_k} \alpha_i Y_i + \sum_{i=k+1}^{m_k} \alpha_i Z_i \le X_k - \sum_{i=k+1}^{m_k} \alpha_i \delta / 2 < \varepsilon / 2,
\end{split}
\end{equation*}
which {leads to contradiction}.
\end{proof}

{
\begin{lemma}\label{lm:supplement}
Assume the step size $\{\alpha_k\}$  satisfy (\ref{Assum1:alpha}),  and its partial sums is $S_n=\sum_{k=1}^n \alpha_k$, then we have 
$$\sum_{k=1}^n \alpha_k / S_k \rightarrow+\infty~as~n\rightarrow+\infty.$$
\end{lemma}
\begin{proof} 
For $\forall m, n>0$, it is easy to get 
$$\sum_{k=m}^n \frac{\alpha_k}{S_k} \geq \frac{S_n-S_m}{S_n}=1-\frac{S_m}{S_n}.$$
Due to $S_n \rightarrow+\infty$, for a given $m$, there exists $n(m)$ and  $\delta\in (0,1)$, such that $S_m / S_n<1-\delta$. Letting $m_k=n\left(m_{k-1}\right)+1$,  we obtain 
$$\sum_{k=1}^{m_K} \frac{\alpha_k}{S_k} \geq \sum_{k=1}^K\left(1-\frac{S_{m_k}}{S_{n\left(m_k\right)}}\right) \geq K \delta \rightarrow+\infty.$$
\end{proof}}

Similar to \cite{Bertsekas00}, we give the convergence result of vSGD below. The difference is that we consider $L^2$ convergence instead of convergence with probability one. The reason why we consider $L^2$ convergence here is that it is more convenient to estimate the convergence rate as  \cite{SGDZO}.

\begin{theorem}\label{SGD:thm1}
{ If the function $f(x)$ is  $L$-smooth and has a lower bound, and the assumptions (\ref{Assum1:alpha}), (\ref{Assum:StatXi}) hold}, then for vSGD we have
\begin{equation}\label{limit1}
\underset{n\rightarrow +\infty}{\lim \inf} \ \mathbb{E}\lVert \nabla f(x_{n})\rVert^2 = 0.
\end{equation}
Furthermore, if $\underset{n\rightarrow \infty}{\lim}\sum_{k=1}^n \alpha_k^2 <\infty$ ,  we have
\begin{equation}\label{limit2}
\underset{n\rightarrow +\infty}{\lim} \ \mathbb{E}\lVert \nabla f(x_{n})\rVert^2 = 0.
\end{equation}
\end{theorem}

\begin{proof}
If $f$ is $L$-smooth, it is easy to get
\begin{equation}\label{eq:th1}
\forall x,y\in \mathbb{R}^d, \quad f(x) \le f(y) + \nabla f(y)^T(x-y) + \frac{L}{2}\| x - y \|^2.
\end{equation}
From (\ref{eq:th1}), (\ref{Assum:StatXi}) and $\mathbb{E} [\xi_k | x_{k-1}] = 0$, it is easy to know that
\begin{equation*}
\begin{split}
\mathbb{E}[f(x_k)|\mathcal{F}_{k-1}] \le &\ f(x_{k-1}) - \alpha_k \| \nabla f(x_{k-1}) \|^2  + \frac{L\alpha_k^2}{2} (||\nabla f(x_{k-1})||^2 + \mathbb{E}[\| \xi_k \|^2|\mathcal{F}_{k-1}]) \\
& \le f(x_{k-1}) - C \alpha_k  \| \nabla f(x_{k-1}) \|^2 + \frac{LM}{2} \alpha_k^2,
\end{split}
\end{equation*}
then we have
\begin{equation} \label{ref:sgditeration}
\begin{split}
\mathbb{E}f(x_k) \le \mathbb{E} f(x_{k-1}) - C \alpha_k \mathbb{E}\lVert \nabla f(x_{k-1}) \rVert^2 + \frac{LM}{2}\alpha_k^2.
\end{split}
\end{equation}

{
Taking  $(X_k,Y_k,Z_k) = (\mathbb{E}f(x_k), C\mathbb{E}\lVert \nabla f(x_{k-1}) \rVert^2, {LM}\alpha_k / 2)$ in Lemma \ref{lm:3sq} (a), we get
\begin{equation*}
\underset{k\rightarrow +\infty}{\lim\inf} \  \mathbb{E} \lVert  \nabla f(x_{k}) \rVert^2 = 0.
\end{equation*}}
{
For $\sum \alpha_k ^2/ 2 < +\infty,$ then by Lemma \ref{lm:3sq} (b), we get $\{ \mathbb{E}f(x_k) \}$ is convergent. If there exists $\varepsilon > 0$, such that $\lim\sup Y_k \ge \varepsilon$, that means we can find infinite $k$ fulfilled $Y_k \le \varepsilon / 4,  Y_{m_k} \ge \varepsilon,  Y_i \in [\varepsilon / 4,\varepsilon], i\in (k,m_k)$, we get
\begin{equation*}
\begin{split}
\mathbb{E}f(x_{m_k}) & \le \mathbb{E}f(x_{k}) + \sum_{i = k+1}^{m_k}\alpha_k Z_k -  K\sum_{i = k+1}^{m_k}\alpha_k Y_k
\end{split}
\end{equation*}
with $k \rightarrow +\infty$, then we have
\begin{equation*}
\begin{split}
\sum_{i = k}^{m_k}\alpha_k\varepsilon/4 \le \sum_{i = k+1}^{m_k}\alpha_k(\mathbb{E}\lVert  \nabla f(x_{k-1}) \rVert^2) \rightarrow 0.
\end{split}
\end{equation*}
With $L$-smooth condition, we obtain
\begin{equation*}
\begin{split}
\mathbb{E} (\lVert \nabla f(x_k) - \nabla f(x_{k-1}) \rVert^2) & \le L \alpha_k^2 \mathbb{E}\lVert  \nabla f(x_{k-1}) - \xi_k \rVert^2 \le C \alpha_k^2 (Y_{k} + 1) \le C \alpha_k^2.
\end{split}
\end{equation*}
By Minkowski inequality,  we have
\begin{equation*}
\begin{split}
\sqrt{\varepsilon}/2 \le Y_k^{1/2} - Y_{m_k}^{1/2} \le C\sum_{i = m+1}^{m_k}\alpha_k \rightarrow 0,
\end{split}
\end{equation*}
which leads to contradiction and the proof of the second part is completed.}

\end{proof}

The following theorem give the convergence result of vSGD when $f(x)$ is convex.

\begin{theorem}
{Assume $f(x)$ is convex and $L$-smooth,  and the assumptions (\ref{Assum1:alpha}), (\ref{Assum:StatXi}) hold.} If $f(x)$ has a lower bound and $\alpha_n \sum_{k=1}^n \alpha_k^2 \rightarrow 0$, then we have
$$
\underset{n\rightarrow +\infty}{\lim} \ \mathbb{E} f(x_n) = f(x^*) = f^*,
$$
where $x^*$  is the minima.
\end{theorem}

\begin{proof} The case for $\sum \alpha_k^2 < +\infty$ has been proved in Theorem \ref{SGD:thm1}. Now we consider $\sum \alpha_k^2 = +\infty.$  Setting $\nu_k = \| x_{k}-x^* \|$, {by \eqref{Assum:StatXi}}, we have
$$
\mathbb{E}\nu_k^2 \le \mathbb{E}\nu_{k-1}^2 - 2\alpha_k \mathbb{E}(x_{k-1} - x^*) \nabla f(x_{k-1}) + (1+V)\alpha_k^2 \mathbb{E}\lVert  \nabla f(x_{k-1})\rVert^2 + M \alpha_k^2.
$$
By the convex condition, we get $(x-x^*)\nabla f(x) > 0$, which gives
\begin{equation}\label{equation11}
\mathbb{E}\nu_k^2 \le \mathbb{E}\nu_{k-1}^2 + (1+V)\alpha_k^2 \mathbb{E} \lVert\nabla  f(x_{k-1}) \rVert^2 + M \alpha_k^2.
\end{equation}
Let $c_n = 1 + \sum_{k=1}^n \alpha_k^2$.  According to Theorem \ref{SGD:thm1}, we know
$
\sum_{k=1}^n \alpha_k \mathbb{E}\lVert \nabla f(x_{k-1})\rVert^2 \le Cc_n,
$
which gives $\mathbb{E}\nu_k^2 \le Cc_n.$
Again with the convex condition, we have
$$(\mathbb{E}f(x_{k-1}))^2 \le \mathbb{E}\lVert \nabla f(x_{k-1}) \rVert^2 \mathbb{E}\nu_{k-1}^2. $$
So
\begin{equation*}
\begin{split}
\mathbb{E}f(x_k) & \le \mathbb{E}f(x_{k-1}) - C\alpha_k \mathbb{E}\lVert \nabla f(x_{k-1}) \rVert^2 + \frac{LM}{2}\alpha_k^2  \\
&\le \mathbb{E} f(x_{k-1}) - C\alpha_k \frac{\mathbb{E} \lVert f(x_{k-1}) \rVert^2}{ \mathbb{E}\nu_{k-1}^2}+ \frac{LM}{2}\alpha_k^2  \\
&\le \mathbb{E} f(x_{k-1}) - C\alpha_kc_{k}^{-1} \mathbb{E} \lVert f(x_{k-1}) \rVert^2 + \frac{LM}{2} \alpha_k^2.
\end{split}
\end{equation*}
With condition $\alpha_n c_n \rightarrow 0,$  taking
$$(X_k, Y_k, Z_k) = (\mathbb{E}f(x_n), C\mathbb{E}\| f(x_{k-1})\|^2, {LM\alpha_kc_{n}/2} )$$
  in Lemma \ref{lm:3sq} and using $\{ \alpha_n c_{n}^{-1} \}$ instead of $\{\alpha_k \}$, it is easy to check that $\sum \alpha_n c_{n}^{-1} = +\infty$ {by Lemma \ref{lm:supplement}}, and we have
\begin{equation*}
\begin{split}
X_k = \mathbb{E}f(x_n) \rightarrow X^* = f^*.
\end{split}
\end{equation*}
\end{proof}

\begin{remark}
The condition $\alpha_n \sum_{k=1}^n \alpha_k^2 \rightarrow 0$ on step size in the above theorem is more general than $\sum \alpha_k^2 < \infty$ in \cite{Bertsekas00}. For example, $\alpha_n = C n^{a}, a \in (1/3,1]$ can still ensure the convergence of vSGD for convex problems.
\end{remark}

\begin{theorem}\label{th:sdg3}
For convex function $f(x)$ with lower bound and the minima $x^*$, if $f(x)$ is $L$-smooth,  the assumptions (\ref{Assum1:alpha}), (\ref{Assum:StatXi}) hold,   and the following condition  (more general than strongly convex)  is fulfilled
\begin{equation}\label{weakconvex}
\begin{split}
\exists K_0, \delta > 0,  (f(x) - f^*)^2 \le K_0\| \nabla f(x)\|^2, \mathrm{for }\ x \in \{\| \nabla f(x)\|^2 \le \delta \},
\end{split}
\end{equation}
then we have
$$
\underset{n\rightarrow +\infty}{\lim} \ \mathbb{E} f(x_n) = f(x^*) = f^*.
$$
\end{theorem}

\begin{proof}
Without loss of generality, we set $f^* = f(x^*) = 0$. From Theorem \ref{SGD:thm1}, we know $\lim \inf_k  \|  \nabla f(x_k)  \|  = 0$. Since $\cap_n \{x|  \| \nabla f(x)  \| < 1/n\} = \{x|  \|  \nabla f(x)  \|  = 0\}$,  we get ${\lim\inf_k \mathbb{E} \lVert f(x_{k}) \rVert^2 = 0} $. Set
\begin{equation*}
Y_k = \begin{cases}
CK_0^{-1}\mathbb{E} \lVert f(x_{k-1}) \rVert^2 &\text{if }  \lVert \nabla f(x_{k-1}) \rVert^2 < \delta,\\
C\mathbb{E} \lVert \nabla f(x_{k-1}) \rVert^2 &\text{otherwise},
\end{cases}
\end{equation*}
and $X_k = \mathbb{E}f(x_n),  Z_k = C\alpha_k$. By (\ref{ref:sgditeration}), according to (c) in Lemma \ref{lm:3sq},  we have
\begin{equation*}
\begin{split}
X_k \rightarrow X^* = f^*.
\end{split}
\end{equation*}
That means we have $f(x_k) \rightarrow f^* $ in the sense of expectation.
\end{proof}

From the above theorem,  it can be seen that the constraints on the step size become less strict by strengthening the conditions of the function $f(x)$.

In fact, the condition (\ref{weakconvex}) is relatively general. For a convex function $f(x)$ with second derivative,
if there exists $ c > 0$ such that $D = \{ x | \|\nabla f(x)\| < c \}$ is bounded, then (\ref{weakconvex}) can be fulfilled. This situation is illustrated by the following example.

Let $u = x - x^*,$ and $h(t) = f(tu) - f(x^*)$.  We have $ h'(t) = \nabla f(tu)^T u \ge 0$ and $h''(t) = u^T{\nabla^2 }f(tu)u \ge 0,$
$$
h(t)^2 = h'(\theta t)^2 t^2 = h'(t)^2 t^2 - (h'(t)^2 - h'(\theta t)^2 )t^2,
$$
where $\theta \in (0,1)$.  With $(h'(t)^2)' = 2h''(t)h'(t)$, we get
$$
 (h'(t)^2 - h'(\theta t)^2 )t^2 =  (1-\theta)h''(\eta t)h'(\eta t)t^3\geq 0,
$$
where $\eta \in (\theta,1)$,  and
$$
h(t)^2 \le  h'(t)^2 t^2, ~(f(x) - f^*)^2 \le \| \nabla f(x)\|^2 \|x - x^*\|^2,
$$
since the set $D$ is bounded. Let $K_0 = \sup_{x\in D}\|x - x^*\|^2$  then (\ref{weakconvex}) is fulfilled. As the following corollary, we can also prove the convergence for strongly convex function.
\begin{corollary}
If $f(x)$ satisfies conditions in Theorem \ref{th:sdg3} and is $\mu$-strongly convex, then $x_k$ converge to $x^*$ in $L^2$.
\end{corollary}
\begin{proof}
Because $f(x) - f(x^*) \le \frac{L}{2} \lVert x - x^* \rVert^2 \le \frac{L}{\mu} \lVert x - x^* \rVert  \lVert \nabla f(x) \rVert $,  the set $D$ exists. The  condition  (\ref{weakconvex}) is fulfilled. So $x_k$ converge to $x^*$ in $L^2$.
\end{proof}

\section{Convergence  of mSGD and NaSGD}\label{sec:msgd}

\subsection{Convergence  of  mSGD}

Assume $f$ is twice differentiable, $L$-smooth (\ref{eq:L-smooth}), and has lower bound $f^*$.
Consider the following mSGD
iteration
\begin{equation}\label{eq:mSGD31}
\begin{aligned}
&x_{k}=x_{k-1}+\alpha_{k} v_{k}, \\
&v_{k}=v_{k-1}-\mu_{k} \alpha_{k} v_{k-1}-\alpha_{k} \nabla f\left(x_{k-1}\right)+\alpha_{k} \xi_{k},
\end{aligned}
\end{equation}
where $\mu_k$ is bounded by $\bar{\mu} $ and $\tilde{\mu}$, i.e., $0 < \bar{\mu} \le \mu_k \le \tilde{\mu}.$  The reason why we consider format (\ref{eq:mSGD31}) is that it is implicitly  mentioned  in \cite{CLTO} and this format also has a continuum limit version like \cite{sirignano2019}.

\begin{theorem}\label{SGD:thm4}
If $f(x)$  has lower bound and the assumptions (\ref{Assum1:alpha}), (\ref{Assum:StatXi}) hold, we have
$$
\underset{n\rightarrow +\infty}{\lim \inf} \ \mathbb{E}\lVert \nabla f(x_{n})\rVert^2 = 0.
$$
If $\underset{n\rightarrow \infty}{\lim}\sum_{k=1}^n \alpha_k^2 <\infty$,  we get
$$
\underset{n\rightarrow +\infty}{\lim} \ \mathbb{E}\lVert \nabla f(x_{n})\rVert^2 = 0.
$$
If $f(x)$ is convex with the minima $x^*$ and satisfies the assumption
\begin{equation}\label{eq:weakvonvex2}
\begin{split}
\exists K_0, \delta > 0, (f(x) - f^*)^2 \le K_0\|\nabla f(x)\|^2, \mathrm{for }\ x \in \{\| \nabla f(x)\|^2 \le \delta \},
\end{split}
\end{equation}
 we have
$$
\underset{n\rightarrow +\infty}{\lim} \ \mathbb{E} f(x_n) = f(x^*) = f^*.
$$
\end{theorem}

\begin{proof}
We first perform Lyaponuv analysis, which is similar to \cite{li2022revisiting}. Define the Hamiltonian $H_k = f(x_k)-f(x^*) + \|v_k\|^2/2 $ and  $\bar{H}_k = \| \nabla f(x_k)\|^2 + \|v_k\|^2$.
By \eqref{eq:mSGD31}, $L$-smoothness of $f$  and $\mathbb{E}[  \| \xi_k \|^2 |\mathcal{F}_{k-1}] \le M + K_{\xi} \| \nabla f(x_{k-1})\|^2$, we have
\begin{equation*}
\begin{split}
\mathbb{E} [f(x_k)|\mathcal{F}_{k-1} ] & \le f(x_{k-1}) + \alpha_k \mathbb{E} [\nabla f(x_{k-1})^T v_k|\mathcal{F}_{k-1} ] + \frac{L\alpha_k^2}{2}\mathbb{E} [ \| v_k\|^2|\mathcal{F}_{k-1} ]   \\
& \le f(x_{k-1}) + \alpha_k\nabla f(x_{k-1})^Tv_{k-1} + C\alpha_k^2 (1 + O (\bar{H}_{k-1}))
\end{split}
\end{equation*}
and
\begin{equation*}
\begin{split}
\mathbb{E}\Big[\frac{\| v_k \|^2}{2}|\mathcal{F}_{k-1}\Big] & \le \frac{\| v_{k-1}\|^2}{2} + \frac{1}{2}\mathbb{E} [ \|v_k - v_{k-1} \|^2|\mathcal{F}_{k-1} ] - \alpha_k \left(\tilde {\mu}\|v_{k-1}\|^2 + \nabla f(x_{k-1})v_{k-1}\right) \\
& \le \frac{\| v_{k-1} \|^2}{2} - \alpha_k (\mu_{k}\|v_{k-1}\|^2 + \nabla f(x_{k-1})v_{k-1}) + C\alpha_k^2(1 + O (\bar{H}_{k-1})),
\end{split}
\end{equation*}
thus
\begin{equation*}
\begin{split}
\mathbb{E} [H_k|\mathcal{F}_{k-1} ] & \le H_{k-1} -  \mu_{k} \alpha_k \| v_{k-1} \|^2 + C\alpha_k^2 (1 +O (\bar{H}_{k-1})).
\end{split}
\end{equation*}

By introducing the term $\tilde{Z}_k = v_k^T\nabla f(x_k)$, with $L$-smoothness of $f$, we get
\begin{equation*}
\begin{split}
\mathbb{E} [\tilde{Z}_k|\mathcal{F}_{k-1} ]  & = \tilde{Z}_{k-1} + \mathbb{E} [ v_{k}^T(\nabla f(x_k) - \nabla f(x_{k-1}))|\mathcal{F}_{k-1} ]+ \mathbb{E}[(v_k - v_{k-1})^T\nabla f(x_{k-1})|\mathcal{F}_{k-1} ] \\
&\le \tilde{Z}_{k-1} + \alpha_k (L \mathbb{E}[ \| v_{k} \|^2 |\mathcal{F}_{k-1} ]  -  \| \nabla f(x_{k-1}) \|^2  - \mu_{k} v_{k-1}^T \nabla f(x_{k-1})) \\
& = \tilde{Z}_{k-1} + \alpha_k (L \| v_{k-1} \|^2  -  \| \nabla f(x_{k-1}) \|^2  - \mu_{k} v_{k-1}^T \nabla f(x_{k-1})) + \alpha_k^2(O(\bar{H}_{k-1})+o(1)) .
\end{split}
\end{equation*}

Now we consider the Lyapunov function $H^E_k = \mathbb{E}\tilde{H}_k$ with $\tilde{H}_k = H_k + \zeta \tilde{Z}_k$, where $\zeta > 0$ is small enough.
Then
\begin{equation*}
\begin{split}
&\mathbb{E} [\tilde{H}_k|\mathcal{F}_{k-1} ]  \le \tilde{H}_{k-1} - \alpha_k F_{k-1} + C \alpha_k^2 (1 + O (\tilde{H}_{k-1})),
\end{split}
\end{equation*}
where $F_{k-1}  =  \zeta\| \nabla f(x_{k-1}) + \frac{\mu_{k}}{2}v_{k-1} \|^2 + (\tilde\mu - \zeta(\frac{\mu_{k}^2}{4} + L)) \| v_{k-1} \|^2$.  We get $\tilde{F}_k=\Theta (\bar{H}_k)$, {where the notation $a_n = \Theta(b_n)$ means there exists $ C_1, C_2 > 0$, such that $C_1 b_n\le |a_n| \le C_2 b_n$.} So there exists $K>0$ such that
\begin{equation}\label{msgd:ite}
\begin{split}
&H^E_k \le H^E_{k-1} - K \alpha_k \mathbb{E}\bar{H}_{k-1} + C \alpha_k^2.
\end{split}
\end{equation}
 From the $L$-smoothness of $f$ and  $f(x)$ with a lower bound $f^*$, we get $\tilde{H}_k\ge \frac{L}{2}\| \nabla f(x) \|^2 + \zeta v_k^T\nabla f(x_k) + \| v_k\|^2 / 2$, which means $\tilde{H}_k, H_k^E$ have lower bound.  Let $V_k = \mathbb{E} \lVert \nabla f(x_{k-1}) \rVert^2 $.
Taking $(X_k, Y_k, Z_k) = (H^E_{k}, \mathbb{E}\bar{H}_{k-1},C \alpha_k)$  in Lemma \ref{lm:3sq} (a),  we get $\lim \inf \bar{H}_k = 0,$  which means $\lim \inf V_k = 0.$

For $\sum \alpha_k^2 < +\infty$, by  in Lemma \ref{lm:3sq} (b), we get $X_k$ is convergent.
If there exists $\varepsilon > 0$, such that
 $\lim\sup_n V_n \ge \varepsilon$  that means  we can find infinite  $k$ fulfilled $V_k \le \varepsilon / 4,  V_{m_k} \ge \varepsilon,  V_k \in [\varepsilon / 4,\varepsilon], i\in (k, m_k)$, we get
\begin{equation*}
\begin{split}
X_{m_k} & \le X_k + \sum_{i = k+1}^{m_k}\alpha_i Z_i -  \sum_{i = k+1}^{m_k}\alpha_i Y_i,
\end{split}
\end{equation*}
then we have
\begin{equation*}
\begin{split}
\sum_{i = k+1}^{m_k}\alpha_i\varepsilon/4 \le \sum_{i = k+1}^{m_k}\alpha_iV_i \le C\sum_{i = k+1}^{m_k}\alpha_i Y_i \rightarrow 0.
\end{split}
\end{equation*}
With $L$-smooth condition,  we obtain
\begin{equation*}
\begin{split}
\mathbb{E} (\lVert \nabla f(x_k) - \nabla f(x_{k-1}) \rVert^2) & \le 2L \alpha_k^2 \mathbb{E}\lVert  v_{k} \rVert^2 \le C(\mathbb{E}\bar{H}_k + 1)\alpha_k^2 \le C\alpha_k^2.
\end{split}
\end{equation*}
By Minkowski inequality, we have
\begin{equation*}
\begin{split}
\sqrt{\varepsilon}/2 \le V_k^{1/2} - V_{m_k}^{1/2} \le C\sum_{i = k+1}^{m_k}\alpha_i \rightarrow 0,
\end{split}
\end{equation*}
which {leads to contradiction} and the proof of the second part is completed.

For the last part with (\ref{eq:weakvonvex2}),  similar to  SGD,  without loss of generality, we set $f^* = f(x^*) = 0$, which means the sharp bound of $\tilde{H}_k$ is $0$. We take
\begin{equation}\label{Yk}
Y_{k+1} = \begin{cases}
K(\mathbb{E} (K_0^{-1}\lVert f(x_{k}) \rVert^2 + \| v_k\|^2) )&\text{if }  \lVert \nabla f(x_{k}) \rVert^2 < \delta,\\
K\mathbb{E} \bar{H}_k &\text{otherwise},
\end{cases}
\end{equation}
$X_k = \mathbb{E}\tilde{H}_k$,  and $Z_k = C\alpha_k$. By (\ref{msgd:ite}), according  to Lemma \ref{lm:3sq} (c), we have
\begin{equation*}
\begin{split}
X_k \rightarrow X^* = 0.
\end{split}
\end{equation*}
This means we have $f(x_k) \rightarrow f^* $  in $L^1$.
\end{proof}


\subsection{Convergence  of NaSGD }

In this section, we give the convergence  of NaSGD, where $\mu_k$ is bounded by $\bar{\mu} $ and $\tilde{\mu}$, i.e., $0 < \bar{\mu} \le \mu_k \le \tilde{\mu}.$  Assume $f$ is twice differentiable and satisfies $L$-smooth (\ref{eq:L-smooth}). Further assume $ \lim \sup \alpha_k/\alpha_{k-1} < +\infty,$ then $\beta_k = (1 - \mu_k \alpha_k)\frac{\alpha_k}{\alpha_{k-1}} $. Let $\hat{\beta} = \lim \sup_k \beta_k.$ Consider the following NaSGD iteration
\begin{equation*}
\begin{aligned}
x_{k} & = x_{k-1} + \alpha_k v_{k}, \\
v_k & = (1 - \mu_k \alpha_k) v_{k-1}  -\alpha_k \nabla f(x_{k-1}+ \beta_k(x_{k-1} - x_{k-2}))  + \alpha_k \xi_k.
\end{aligned}
\end{equation*}

The following theorem presents the convergent results under $L$-smooth (\ref{eq:L-smooth}) and a more general learning rates conditions, which is previously unknown.

\begin{theorem}\label{LfSGD:thm1}
{If $f(x)$  has lower bound $f^*$,  the assumptions (\ref{Assum1:alpha}), (\ref{Assum:StatXi})  hold,} and  $L\hat{\beta} < \bar{\mu}$ or $f(x)$ is convex, then we have
$$
\underset{n\rightarrow +\infty}{\lim \inf} \ \mathbb{E}\lVert \nabla f(x_{n})\rVert^2 = 0.
$$
If $\underset{n\rightarrow \infty}{\lim}\sum_{k=1}^n \alpha_k^2 <\infty$, then we get
$$
\underset{n\rightarrow +\infty}{\lim} \ \mathbb{E}\lVert \nabla f(x_{n})\rVert^2 = 0.
$$
If $f(x)$ is convex and satisfies the assumption
\begin{equation*}
\begin{split}
\exists K_0,  \delta > 0, (f(x) - f^*)^2 \le K_0\|\nabla f(x)\|^2, \mathrm{for }\ x \in \{\| \nabla f(x)\|^2 \le \delta \},
\end{split}
\end{equation*}
then we obtain
$$
\underset{n\rightarrow +\infty}{\lim} \ \mathbb{E} f(x_n) = f(x^*) = f^*.
$$
\end{theorem}
\begin{proof}
For  $L\hat{\beta} < \bar{\mu}$,   consider the Hamiltonian function
$H_k = f(x_k) - f^* + \|v_k\|^2/2 $. Let $\bar{H}_k = \| \nabla f(x_k)\|^2 + \|v_k\|^2$.  We have
\begin{equation*}
\begin{split}
\mathbb{E} [H_k|\mathcal{F}_{k-1} ] & \le H_{k-1} +\alpha_k [v_{k-1}^T(\nabla f(x_{k-1}) - \nabla f(y_{k-1})) - \mu_k \| v_{k-1}\|^2 ] + o(\alpha_k ) \\
& \le H_{k-1} -  (\bar{\mu} - L\hat{\beta}) \alpha_k  \| v_{k-1}\|^2 + C\alpha_k^2(1 +O (\bar{H}_{k-1})).
\end{split}
\end{equation*}
When $f$ is convex, we have $v_{k-1}^T(\nabla f(x_{k-1}) - \nabla f(y_{k-1})) < 0,$ and the same result can be obtained except that the coefficient of the $\| v_{k-1}\|^2$ term in the above formula is $-\bar{\mu}$.

By introducing the term $\tilde{Z}_k =  v_k^T\nabla f(x_k) $, we get
\begin{equation*}
\begin{split}
\mathbb{E} [\tilde{Z}_k|\mathcal{F}_{k-1} ] &\le Z_{k-1} +   \alpha_k (L\| v_{k-1}\|^2 -   \nabla f(x_{k-1})^T\nabla f(y_{k-1}) - \tilde{\mu} v_{k-1}^T \nabla f(x_{k-1}) )+ \alpha_k^2 O (\bar{H}_{k-1}).
\end{split}
\end{equation*}
From $L$-smoothness of $f$, we have
\begin{equation*}
|\nabla f(x_{k-1})^T(\nabla f(x_{k-1})- \nabla f(y_{k-1}))| \le L\beta_k | v_{k-1}^T \nabla f(x_{k-1})| \le \hat{\beta}L \lambda \| \nabla f(x_{k-1}) \|^2 + \frac{\hat{\beta}L}{\lambda} v_{k-1}^2.
\end{equation*}
Set $\lambda= 1/(2L\hat{\beta})$, then we obtain
\begin{equation*}
\begin{split}
\mathbb{E} [\tilde{Z}_k|\mathcal{F}_{k-1} ] &\le Z_{k-1} +  \alpha_k \Big((L+ \frac{L\hat{\beta}}{\lambda})\| v_{k-1}\|^2 -  (1 - L\hat{\beta}\lambda) \| \nabla f(x_{k-1}) \|^2 - \tilde{\mu} v_{k-1}^T \nabla f(x_{k-1})\Big)\\
&+ C\alpha_k^2(1 + O (\bar{H}_{k-1})).
\end{split}
\end{equation*}
Now we consider the Lyapunov function $H^E_k = \mathbb{E}\tilde{H}_k=\mathbb{E}(H_k + \zeta \tilde{Z}_k)$ with small enough $\zeta>0$. Similar to the analysis of mSGD,  there exists $K>0$ such that
\begin{equation*}
\begin{split}
&H^E_k \le H^E_{k-1} - K \alpha_k \mathbb{E}\bar{H}_{k-1} + C \alpha_k^2.
\end{split}
\end{equation*}
The rest analysis is similar to mSGD, { that is, using the lemma \ref{lm:3sq} to obtain the convergence. We take $(X_k, Y_k, Z_k) = (H^E_{k}, \mathbb{E}\bar{H}_{k-1},C \alpha_k)$ in Lemma \ref{lm:3sq} (a), 
 (b) and for the last part we change $Y_k$ as \eqref{Yk} and use Lemma \ref{lm:3sq} (c), to obtain the convergence with convexity.} So we omit it.
\end{proof}

\subsection{Convergence  of mSGD with vanishing damping $\mu_k \rightarrow 0$}
In this section, we will give the convergence  of mSGD with vanishing damping $\mu_k \rightarrow 0$.
Assume $f$ is twice differentiable.
For the vanishing damping case, we need to make some modifications to (\ref{Assum1:alpha}). The assumption corresponding to the divergence condition (\ref{Assum1:alpha}) of mSGD is
\begin{equation}\label{eq:divergence2}
\begin{split}
&\lim_{k\rightarrow \infty}\alpha_k=0,
\quad
\lim_{k\rightarrow \infty}\mu_k=0,
\quad
\lim_{k\rightarrow \infty}\frac{\alpha_k}{\mu_k}=0,
\quad\sum_{k=1}^\infty \alpha_k\mu_k=\infty,\\
& \exists L_{\mu} \geq 0,\quad \mu_{k-1}-\mu_{k}=L_{\mu} \alpha_{k} \mu_{k}+o\left(\alpha_{k} \mu_{k}\right).
\end{split}
\end{equation}

\begin{theorem}\label{NaSGD:thm1}
Suppose that function $f(x)$  is $L$-smooth and has  minima $x^*$. If
(\ref{Assum1:alpha}), (\ref{Assum:StatXi})  hold and $\mu_k$  satisfy (\ref{eq:divergence2}), then we have
$$
\underset{n\rightarrow +\infty}{\lim \inf} \ \mathbb{E}\lVert \nabla f(x_{n})\rVert^2 = 0.
$$

If $f(x)$ is convex and satisfies the assumption
\begin{equation*}
\begin{split}
\exists K_0, \delta > 0, (f(x) - f^*)^2 \le K_0\|\nabla f(x)\|^2, \mathrm{for }\ x \in \{\| \nabla f(x)\|^2 \le \delta \},
\end{split}
\end{equation*}
then we have
$$
\underset{n\rightarrow +\infty}{\lim} \ \mathbb{E} f(x_n) = f(x^*) = f^*.
$$
\end{theorem}

\begin{proof}

For mSGD with $\mu_k \rightarrow 0$, consider the Lyapunov function $H^E_k = \mathbb{E}\tilde{H}_k$, where $\tilde{H}_k = H_k + \lambda \mu_k \tilde{Z}_k,$ with $0<\lambda < 1/L$. Similar to mSGD, we have
\begin{equation*}
\begin{split}
\mathbb{E} [\tilde{H}_k|\mathcal{F}_{k-1} ]  &\le \tilde{H}_{k-1} - \alpha_k \tilde{F}_{k-1} - \lambda(\mu_k - \mu_{k-1}) v_{k-1}^T \nabla f(x_{k-1})  +  C(1+O (\bar{H}_{k-1}))\alpha_k^2,
\end{split}
\end{equation*}
where $
\tilde{F}_{k-1} = \lambda \mu_k (\| \nabla f(x_{k-1}) \|^2 + v_{k-1}^T \nabla f(x_{k-1})) + (1 - L \lambda)\mu_k \| v_{k-1} \|^2.$

Further, let $\lambda < (L + L_\mu^2 / 4)^{-1}$. By (\ref{eq:divergence2}), we have $\tilde{F}_k + \lambda(\mu_{k+1} - \mu_{k})\tilde{Z}_{k} = \mu_{k+1} \Theta(\bar{H}_k)$.  Similar to the analysis of mSGD, there exists $K>0$ such that
\begin{equation*}
\begin{split}
\mathbb{E} [\tilde{H}_k|\mathcal{F}_{k-1} ]  &\le \tilde{H}_{k-1} - K \alpha_k \mu_k \mathbb{E}\bar{H}_{k-1} + C(1+O (\tilde{H}_{k-1}))\alpha_k^2. \\
\end{split}
\end{equation*}
Then we have
 $$H^E_k \le \tilde{H}_{k-1} - K \alpha_k\mu_k \mathbb{E} \bar{H}_{k-1} + C\alpha_k^2.$$
Take $(X_k, Y_k, Z_k) = (H^E_{k}, \mathbb{E}\bar{H}_{k-1}, C \alpha_k / \mu_k)$, and use $\{ \alpha_k \mu_k\}$ instead of $\{\alpha_k \}$ in Lemma \ref{lm:3sq} { to get convergence, and for the last part we change $Y_k$ to \eqref{Yk} and use Lemma \ref{lm:3sq} (c), to obtain the convergence with convexity}. The rest analysis is similar to mSGD,  so we omit it.
\end{proof}

\begin{remark} The conditions in (\ref{eq:divergence2}) about  $\{ \alpha_k \}$  and $\{ \mu_k \}$ are reasonable.  For specific examples, see \cite{li2022revisiting}.
\end{remark}

\section{Convergence  of Average SGD} \label{sec:average}

The time average SGD we consider in this section differs from the previous
average  $\hat{x}_n = \sum_k x_k / n$ in \cite{ASGD}.  The form $\bar{x}_n$  has better properties in convex problems, which keeps the convergence $\mathbb{E}(f(\bar{x}_n))\Longrightarrow \mathbb{E}(f(x^*))$ automatically.  The previous researches, such as {\cite{ASSGD}},  need $\nabla f(x)$ to be bounded to ensure the convergence. We can use $L$-smooth to replace this.

\begin{theorem}\label{SGD:thm7}
If the convex function $f(x)$ is $L$-smooth and has a minima $x^*$,  and (\ref{Assum1:alpha}) and (\ref{Assum:StatXi}) hold, then for Average SGD we have
$$
\underset{n\rightarrow +\infty}{\lim} \ \mathbb{E} f(\bar{x}_n) = f(x^*) = f^*,
$$
where $x^*$  is the minima.
\end{theorem}

\begin{proof}
For  $L$-smooth convex function $f(x)$  with minima $x^*$, we have
$$f(\bar{x}_n) \le \frac{\sum_{k=1}^n \alpha_k (f(x_k) - f(x^*))}{\sum_{k=1}^n \alpha_k },$$
and
\begin{equation}\label{average1}
\begin{split}
 \alpha_k (\mathbb{E}f(x_k) - f(x^*)) \le \mathbb{E} \lVert x_{k-1} - x^* \rVert ^2 - \mathbb{E} \lVert x_{k} - x^* \rVert^2 + C \alpha_k^2 (1+  \mathbb{E}\lVert \nabla f(x_k)\rVert)^2).
\end{split}
\end{equation}
From Theorem \ref{SGD:thm1} and (\ref{ref:sgditeration}), we can get $\sum \alpha_k \mathbb{E}\lVert \nabla f(x_k)\rVert^2 < +\infty$ .
With Assumptions (\ref{Assum1:alpha}),  we obtain $\sum \alpha_k^2 \mathbb{E}\lVert \nabla f(x_k)\rVert^2 < +\infty.$

By summing both sides of (\ref{average1}) at the same time, we  get
\begin{equation*}
\begin{split}
\sum_{k=1}^n  \alpha_k (\mathbb{E} f(x_k) - f(x^*)) \le C(1 + \sum_{k=1}^n \alpha_k^2) - \mathbb{E}\lVert x_{n} - x^* \rVert^2,
\end{split}
\end{equation*}
dividing both sides of the above equation by $\sum_{k=1}^n \alpha_k$, we have
\begin{equation*}
\begin{split}
\mathbb{E}f(\bar{x}_n) - f(x^*) \le C\frac{1 + \sum_{k=1}^n \alpha_k^2}{\sum_{k=1}^n \alpha_k} \rightarrow 0.
\end{split}
\end{equation*}
\end{proof}

\section{Conclusion}
In this article,  we studied the convergence of the vSGD method under  more general learning rates conditions and a more general convex assumption.   We also investigated the convergence of the mSGD and NaSGD method with usual damping $\mu_k$ and vanishing damping
$\mu_k\rightarrow 0$ by taking advantage of the Lyapunov function technique, which has been less studied.
 The convergence of time averaged SGD  was also analyzed. The application and further extension of the results obtained in this paper will be the next work.
\section*{Acknowledgments}
{The authors thank Professor Tiejun Li for his constructive suggestions. The authors also appreciate the valuable comments of the referees.} This work  was supported by the NSFC under grant no. 11825102 and National Key R\&D Program of China under grant No. 2021YFA1003300.

\newpage

\bibliographystyle{abbrv}
\bibliography{CON-Ref}

\begin{thebibliography}{10}

\bibitem{NASGDB}
M.~Assran and M.~Rabbat.
\newblock On the convergence of {N}esterov's accelerated gradient method in
  stochastic settings.
\newblock In {\em Proceedings of the 37th International Conference on Machine
  Learning}, volume 119 of {\em Proceedings of Machine Learning Research},
  pages 410--420, 13--18 Jul 2020.

\bibitem{MGDC}
A.~Barakat and P.~Bianchi.
\newblock Convergence rates of a momentum algorithm with bounded adaptive step
  size for nonconvex optimization.
\newblock In {\em Proceedings of The 12th Asian Conference on Machine
  Learning}, volume 129 of {\em Proceedings of Machine Learning Research},
  pages 225--240, 18--20 Nov 2020.

\bibitem{Bertsekas00}
D.~P. Bertsekas and J.~N. Tsitsiklis.
\newblock Gradient convergence in gradient methods with errors.
\newblock {\em SIAM Journal on Optimization}, 10(3):627--642, 2000.

\bibitem{bottou2012stochastic}
L.~Bottou.
\newblock Stochastic gradient descent tricks.
\newblock In {\em Neural networks: Tricks of the trade}, pages 421--436.
  Springer, 2012.

\bibitem{BLN}
L.~Bottou, F.~E. Curtis, and J.~Nocedal.
\newblock Optimization methods for large-scale machine learning.
\newblock {\em SIAM Review}, 60(2):223--311, 2018.

\bibitem{CLTO}
H.~Chen.
\newblock {\em Stochastic Approximation and Its Applications}.
\newblock Kluwer Academic Press, New York, 2003.

\bibitem{SGDZO}
S.~Ghadimi and G.~Lan.
\newblock Stochastic first- and zeroth-order methods for nonconvex stochastic
  programming.
\newblock {\em SIAM Journal on Optimization}, 23(4):2341--2368, 2013.

\bibitem{gitman2019understanding}
I.~Gitman, H.~Lang, P.~Zhang, and L.~Xiao.
\newblock Understanding the role of momentum in stochastic gradient methods.
\newblock In {\em Advances in Neural Information Processing Systems},
  volume~32. Curran Associates, Inc., 2019.

\bibitem{Goodfellow2016}
I.~Goodfellow, Y.~Bengio, and A.~Courville.
\newblock {\em Deep Learning}.
\newblock MIT Press, 2016.

\bibitem{SVRG}
R.~Johnson and T.~Zhang.
\newblock Accelerating stochastic gradient descent using predictive variance
  reduction.
\newblock In {\em Advances in Neural Information Processing Systems},
  volume~26. Curran Associates, Inc., 2013.

\bibitem{li2022revisiting}
T.~Li, T.~Xiao, and G.~Yang.
\newblock Revisiting the central limit theorems for the sgd-type methods.
\newblock {\em arXiv preprint arXiv:2207.11755}, 2022.

\bibitem{SGDB}
Y.~Liu, Y.~Gao, and W.~Yin.
\newblock An improved analysis of stochastic gradient descent with momentum.
\newblock In {\em Advances in Neural Information Processing Systems},
  volume~33, pages 18261--18271. Curran Associates, Inc., 2020.

\bibitem{SGDASGD}
E.~Moulines and F.~Bach.
\newblock Non-asymptotic analysis of stochastic approximation algorithms for
  machine learning.
\newblock In {\em Advances in Neural Information Processing Systems},
  volume~24. Curran Associates, Inc., 2011.

\bibitem{1983A}
Y.~E. Nesterov.
\newblock A method of solving a convex programming problem with convergence
  rate $o(1/k^2)$.
\newblock In {\em Doklady Akademii Nauk}, volume 269, pages 543--547. Russian
  Academy of Sciences, 1983.

\bibitem{nguyen2018sgd}
L.~Nguyen, P.~H. Nguyen, M.~Dijk, P.~Richt{\'a}rik, K.~Scheinberg, and
  M.~Tak{\'a}c.
\newblock Sgd and hogwild! convergence without the bounded gradients
  assumption.
\newblock In {\em International Conference on Machine Learning}, pages
  3750--3758. Proceedings of Machine Learning Research, 2018.

\bibitem{SARAH}
L.~M. Nguyen, J.~Liu, K.~Scheinberg, and M.~Tak{\'a}{\v{c}}.
\newblock {SARAH}: A novel method for machine learning problems using
  stochastic recursive gradient.
\newblock In {\em Proceedings of the 34th International Conference on Machine
  Learning}, volume~70 of {\em Proceedings of Machine Learning Research}, pages
  2613--2621, 06--11 Aug 2017.

\bibitem{JMLR:v20:18-759}
L.~M. Nguyen, P.~H. Nguyen, P.~Richt{{\'a}}rik, K.~Scheinberg,
  M.~Tak{{\'a}}{\v{c}}, and M.~van Dijk.
\newblock New convergence aspects of stochastic gradient algorithms.
\newblock {\em Journal of Machine Learning Research}, 20(176):1--49, 2019.

\bibitem{2019Finite}
L.~M. Nguyen, M.~van Dijk, D.~Phan, P.~H. Nguyen, T.-W. Weng, and
  J.~Kalagnanam.
\newblock Finite-sum smooth optimization with sarah.
\newblock {\em Computational Optimization and Applications}, 82:561 -- 593,
  2019.

\bibitem{1964Some}
B.~T. Polyak.
\newblock Some methods of speeding up the convergence of iteration methods.
\newblock {\em Ussr Computational Mathematics \& Mathematical Physics},
  4(5):1--17, 1964.

\bibitem{ASGD}
B.~T. Polyak and A.~B. Juditsky.
\newblock Acceleration of stochastic approximation by averaging.
\newblock {\em SIAM Journal on Control and Optimization}, 30(4):838--855, 1992.

\bibitem{NCSGDP1}
S.~J. Reddi, A.~Hefny, S.~Sra, B.~Poczos, and A.~Smola.
\newblock Stochastic variance reduction for nonconvex optimization.
\newblock In {\em Proceedings of The 33rd International Conference on Machine
  Learning}, volume~48 of {\em Proceedings of Machine Learning Research}, pages
  314--323, New York, USA, 20--22 Jun 2016.

\bibitem{ASSGD}
O.~Sebbouh, R.~M. Gower, and A.~Defazio.
\newblock Almost sure convergence rates for stochastic gradient descent and
  stochastic heavy ball.
\newblock In {\em Proceedings of Thirty Fourth Conference on Learning Theory},
  volume 134 of {\em Proceedings of Machine Learning Research}, pages
  3935--3971, 15--19 Aug 2021.

\bibitem{sirignano2019}
J.~Sirignano and K.~Spiliopoulos.
\newblock Stochastic gradient descent in continuous time: A central limit
  theorem.
\newblock {\em Stochastic Systems}, 10(2):124--151, 2020.

\bibitem{NAG1}
W.~Su, S.~Boyd, and E.~J. Cand{{\`e}}s.
\newblock A differential equation for modeling nesterov's accelerated gradient
  method: Theory and insights.
\newblock {\em Journal of Machine Learning Research}, 17(153):1--43, 2016.

\end{thebibliography}
\end{document}